\documentclass[reqno]{amsart}
\usepackage{amsfonts,amssymb}
\usepackage[latin1]{inputenc}
\usepackage{enumerate}

\usepackage{latexsym}
\usepackage{textcomp}
\usepackage{amsmath}
\usepackage{amssymb}
\usepackage{mathrsfs}
\usepackage{dsfont}

\newtheorem{thm}{Theorem}[section]
\newtheorem{lem}[thm]{Lemma}
\newtheorem{prop}[thm]{Proposition}
\newtheorem{cor}[thm]{Corollary}

\theoremstyle{definition}
\newtheorem*{notation}{Notation}

\renewcommand{\le}{\leqslant}
\renewcommand{\leq}{\leqslant}
\renewcommand{\ge}{\geqslant}
\renewcommand{\geq}{\geqslant}

\newcommand{\R}{\mathbb{R}}
\newcommand{\N}{\mathbb{N}}

\newcommand{\be}{\begin{enumerate}}
\newcommand{\ee}{\end{enumerate}}

\newcommand{\mC}{\mathcal{C}}
\newcommand{\mF}{\mathcal{F}}

\newcommand{\mN}{\mathcal{N}}

\newcommand{\ds}{\displaystyle}

\newcommand{\ps}[2]{\langle\,#1,#2\,\rangle}

\newcommand{\what}[1]{\widehat{#1}}

\newcommand{\llim}{\lim\limits}
\newcommand{\infl}{\inf\limits}

\newcommand{\suml}{\sum\limits}

\newcommand{\embed}{\hookrightarrow}
\newcommand{\weakto}{\rightharpoonup}
\newcommand{\Ds}{(-\Delta)^s}

\begin{document}

\title[Nonlinear fractional Schr\"odinger equation]{Positive solutions to some nonlinear fractional Schr\"odinger equations via a Min-Max procedure}
\author{Gilles \'Ev\'equoz}
\address{\ Institut f\"ur Mathematik, Johann Wolfgang Goethe-Universit\"at,
Robert-Mayer-Str. 10, 60054 Frankfurt am Main, Germany}
\email{evequoz@math.uni-frankfurt.de}
\author{Mouhamed Moustapha Fall}
\address{\ African Institute for Mathematical Sciences (A.I.M.S.) of Senegal, KM 2, Route de Joal, B.P. 1418 Mbour, S\'en\'egal }
\email{mouhamed.m.fall@aims-senegal.org}
\date{\today}

\begin{abstract}
The existence of a positive solution to the following fractional semilinear equation is proven, in a situation where a 
ground state solution may not exist.
More precisely, we consider for $0<s<1$ the equation
\begin{equation*}
 (-\Delta)^s u + V(x) u = Q(x)|u|^{p-2}u \quad  \textrm{ in } \R^N,\ N\geq 1,
\end{equation*}
where the exponent $p$ is superlinear but subcritical, and $V>0$, $Q\geq 0$ are bounded functions converging to
$1$ as $|x|\to\infty$. Using a min-max procedure introduced by Bahri and Li we prove the existence of a 
positive solution under one-sided asymptotic bounds for $V$ and $Q$.
\end{abstract}
\maketitle

\section{Introduction and main result}
Recently Laskin, in  \cite{L1,L2}, derived an expansion of the Feynman path integral from Brownian-like 
to L\'evy-like quantum mechanical paths
yielding the fractional Schr\"{o}dinger equation:
\begin{equation}\label{eq:1}
\imath \frac{\partial \psi}{\partial t} = \Ds \psi + V(x) \psi ,
\end{equation}
where $(x,t) \in \mathbb{R}^N \times (0,+\infty)$, $0<s\leq 1$, and $V \colon \mathbb{R}^N \to \mathbb{R}$ 
is an external potential function.
When $s=1$, the L\'evy dynamics becomes the Brownian dynamics, and
equation (\ref{eq:1}) reduces to the classical Schr\"{o}dinger equation
\[
\imath \frac{\partial \psi}{\partial t} = -\Delta \psi + V(x) \psi.
\]
Standing wave solutions to  the fractional Schr\"{o}dinger equation  are solutions of the form
\begin{equation}\label{eq:2}
\psi(x,t)=\mathrm{e}^{-\imath   t} u(x),
\end{equation}
where $u$ solves the elliptic equation
\[
(-\Delta)^s u + V(x) u =0,
\]
for $s\in(0,1]$.
The  Schr\"{o}dinger equation with nonlinear source term has also its own interest, especially, when dealing with relativistic
particles ($s=1/2$), see for example \cite{AT, W87, bona97, ann97, maris, KMR, FL2013}.\\
In the present paper, we consider the following nonlinear  fractional Schr\"odinger equation
\begin{equation}\label{eqn:frac_schroed}
(-\Delta)^su +V(x) u = Q(x)|u|^{p-2}u\quad\text{in }\R^N,
\end{equation}
with $0<s<1$ and coefficients satisfying
\newpage
\begin{itemize}
  \item[(V)] $V\in L^\infty(\R^N)$, $\operatorname*{ess\ inf}\limits_{\R^N}V>0$ and  $V(x)\to 1$ as $|x|\to\infty$.

  \item[(Q)] $Q\in L^\infty(\R^N)$, $Q\geq 0$ a.e. on $\R^N$ and $Q(x)\to 1$ as $|x|\to\infty$.
\end{itemize}
Notice that the condition on the value of the asymptotic limit of $V$ and $Q$ is not restrictive, since the case $V(x)\to V_\infty>0$ and
$Q(x)\to Q_\infty>0$ can be treated by means of a rescaling.

We now state our main result:
\begin{thm}\label{thm:exist}
  Let $N\geq 1$, $0<s<1$ and $2<p< \frac{2N}{N-2s}$ for $N >2s$ and $p\in (2,\infty)$ for $2s\geq N=1$. 
  Suppose (V), (Q) and the following conditions hold:
  \begin{center}
  {\rm (H)} There exists $\kappa_1, \kappa_2\geq 0$ and $\alpha>N+2s$ such that for almost every $x\in\R^N$
  \begin{equation*}
    V(x)\leq 1 + \kappa_1 |x|^{-\alpha}\text{ and }\quad Q(x)\geq 1-\kappa_2 |x|^{-\alpha}.
  \end{equation*}
  \end{center}
  
  \begin{center}
  {\rm (U)} The limit problem $(-\Delta)^su+u=|u|^{p-2}u$ in $\R^N$ has (up to translations) a unique positive solution.
  \end{center}
Then \eqref{eqn:frac_schroed} has a positive solution.
\end{thm}
The assumption (U) has been shown to hold in the physically relevant special case of the Benjamin-Ono equation, 
see \cite{AT}, where uniqueness was proved for $s=\frac12$, $N=1$ and $p=2$. In the general case $s\in(0,1)$, by now only 
the ground-state (i.e. least energy) solution of the limit problem is known to be unique (up to translations), see \cite{fv2013}, \cite{FL2013} and \cite{FLS2013}, 
and therefore the assumption (U) needs to be imposed. In the local case $s=1$, the celebrated result by Kwong \cite{Kwong} shows that the limit
problem has a unique positive solution, but in the fractional setting, the question is much more involved, since ODE methods
are not directly applicable anymore. Nevertheless, it is expected that the uniqueness of the positive solution of the limit problem 
holds for most $s\in(0,1)$.

Concerning the regularity of the solutions given by the preceding theorem, 
we note that a standard bootstrap argument based on Sobolev embeddings together with 
\cite[Lemma B.1 and Proposition B.3]{FLS2013} implies that $u\in H^{2s}(\R^N)\cap L^\infty(\R^N)\cap C^{0,\alpha}(\R^N)$
for any $0<\alpha<2s$. In particular, $|u(x)|\to 0$, as $|x|\to\infty$.

Since only recently, intensive work has  been devoted  to the study of \eqref{eqn:frac_schroed} for all $s\in(0,1)$, 
also with general subcritical right-hand side of the form $f(x,u)$.
In  \cite{FQT2013}, existence, asymptotic behavior and symmetry properties of the solutions were studied.
In \cite{Cho} the author gave an existence result for \eqref{eqn:frac_schroed} with $V=1$ while $Q$ was assumed to be 
positive only in a set of positive measure.
In  \cite{Secchi2}, several existence results were proved for problem \eqref{eqn:frac_schroed} with 
more general nonlinearities on the right hand side, generalizing \cite{Cheng}.
For related works about existence and qualitative of solutions, one can also 
see \cite{dipierro}, \cite{fv2013}, \cite{Feng}, \cite{FLS2013} and \cite{Secchi1}.

Our main result in Theorem \ref{thm:exist} is not contained in the afore mentioned papers  
because of the mild assumptions on $V$ and $Q$. In addition, in our  proof, the tools involved
to get positive solution are different.
Indeed, our scheme of proof is based on min-max arguments in the spirit 
of \cite{Bahri-Coron, BaLi90, BaLions97} applied to the energy functional associated with
\eqref{eqn:frac_schroed} on the Nehari manifold (or natural constraint).
The polynomial bounds on $V$ and $Q$ are reminiscent of the exponential 
decay estimates assumed in \cite{BaLi90} and \cite{BaLions97}, and
are indeed related to the asymptotic behavior of ground states for the limit 
problem $\Ds u+u= u^p$. When $s=1$ ground states decay exponentially while for $s\in (0,1)$ they 
decay polynomially (and never exponentially, see \cite{FQT2013}!).
We point out that, since $\Ds +V$ is positive definite we can work with the 
usual Nehari manifold, in contrast to \cite{ew12} where
an indefinite problem was treated using the generalized Nehari set.

Finally, we would like to point out that in Theorem \ref{thm:exist} there is no restriction 
in assuming that the exponent $\alpha$ in condition $(H)$ satisfies
$$
 N+2s<\alpha<\min\{2(N+2s), \frac{p}{2}(N+2s)\},
$$
since $p>2$.

The paper is organized as follows. In Section \ref{sec:prelim} we set up the variational framework for 
the study of \eqref{eqn:frac_schroed}. In particular we show
that the energy of any weak solution of \eqref{eqn:frac_schroed} which changes sign must be higher 
than twice the ground state energy.
In Section~\ref{sec:asympt_estim}, we give a splitting result for Palais-Smale sequences in 
the spirit of \cite{benci-cerami87}, and prove the main energy 
estimate which will allow us in Section \ref{sec:existence_proof} to define, using a barycenter 
map, a min-max critical level and complete the proof of Theorem~\ref{thm:exist}.

\section{Preliminaries}\label{sec:prelim}
Throughout this paper we shall use the following notation. For a function $u$ on $\R^N$ and an element $y\in\R^N$,
we write $y\ast u$ for the translate of $u$ by $y$, {\em i.e.},
\[(y\ast u)(x):=u(x-y),\quad x\in\R^N.\]
Since condition (V) holds, the spectrum of the unbounded operator $S=(-\Delta)^s+V$ acting in $L^2(\R^N)$ is contained
in $(0,\infty)$ and, therefore, $$\|u\|_V:=\left(\int_{\R^N}|\xi|^{2s}|\what{u}(\xi)|^2\, d\xi + \int_{\R^N}V(x)|u(x)|^2\, dx\right)^{\frac12}$$
defines a norm on $H^s(\R^N)$, equivalent to the standard norm 
$$\|u\|_s=\left(\int_{\R^N}(|\xi|^{2s}+1)|\what{u}(\xi)|^2\, d\xi\right)^{\frac12}.$$
The weak solutions of \eqref{eqn:frac_schroed} are critical points of the energy functional $J$: $H^s(\R^N)$ $\to$ $\R$ given by
\[J(u)= \frac{1}{2}\|u\|_V^2 - \frac1p\int_{\R^N}Q(x)|u(x)|^p\, dx, \quad u\in H^s(\R^N).\]
We consider the Nehari manifold (or natural constraint)
\[\mN=\{u\in H^s(\R^N)\backslash\{0\}\, :\, J'(u)(u)=0\}\]
 which contains all critical points of $J$ and set
\begin{equation}\label{eq:c}
c=\inf_{u\in\mN}J(u).
\end{equation}
For the limit problem
\begin{equation}\label{eqn:frac_schroed_inf}
(-\Delta)^su + u = |u|^{p-2}u\quad\text{in }\R^N,
\end{equation}
we denote by $J_\infty$ the associated energy functional given by
\[J_\infty(u)=\frac12\|u\|_s^2 - \frac1p\int_{\R^N}|u|^p\, dx,\quad u\in H^s(\R^N),\]
consider the corresponding Nehari manifold
\[\mN_\infty=\{u\in H^s(\R^N)\backslash\{0\}\,:\, J_\infty'(u)u=0\}\]
and let
$$
c_\infty=\infl_{u\in\mN_\infty}J_\infty(u).
$$
We start by giving some properties of the Nehari manifold $\mN$ and study the behavior of $J$ on
it. Some of the following results can be found in \cite{WILLEM}.
\begin{lem}\label{lem:Nehari_f}
  Under the conditions (V) and (Q), the following holds:
  \begin{itemize}
    \item[(i)] $\mN\neq\emptyset$. More precisely, for every $u\in H^s(\R^N)$ with $\int_{\R^N}Q(x)|u(x)|^p\, dx>0$,
    there exists a unique $t_u\in(0,\infty)$, given by
    \begin{equation}\label{eqn:t_u}
    t_u=\left(\frac{\|u\|_V^2}{\int_{\R^N}Q(x)|u(x)|^p\, dx}\right)^{\frac{1}{p-2}},
    \end{equation}
    such that $t_uu\in\mN$ and $J(t_uu)>J(t u)$ for all $t\geq 0$, $t\neq t_u$. Consequently,
    \begin{equation}\label{eqn:c_alt}
    c=\infl_{u\in\mN}J(u)=\infl_{\substack{u\in H^s(\R^N)\\ u\neq 0}}\sup\limits_{t>0}J(tu).
    \end{equation}
    \item[(ii)] $c>0$, $\inf\limits_{u\in\mN}\|u\|_V>0$ and $\inf\limits_{u\in\mN}\|u\|_{L^p}>0$.
    \item[(iii)] $J$ is coercive on $\mN$, i.e. if $(u_n)_n\subset\mN$ satisfies $\|u_n\|_V\to\infty$ as $n\to\infty$, 
    then $J(u_n)\to \infty$ as $n\to\infty$.
  \end{itemize}
\end{lem}
\begin{proof}
(i) Let us consider $u\in H^s(\R^N)$ satisfying $\int_{\R^N}Q(x)|u(x)|^p\, dx>0$. 
(Note that the assumption (Q) gives some $r>0$ such that $Q(x)\geq \frac12$ 
for a.e. $x$ with $|x|\geq r$, thereby ensuring the existence of such functions $u$.)
 For $t>0$ there holds
$$
\frac{d}{dt}J(tu)=J'(tu)u=t\left(\|u\|_V^2-t^{p-2}\int_{\R^N}Q(x)|u(x)|^p\, dx\right).
$$
Setting $t_u>0$ as in \eqref{eqn:t_u}, we find that the map $t\mapsto J(tu)$ is strictly increasing for $0<t<t_u$ 
and strictly decreasing for $t>t_u$.
Thus, $t_uu\in\mN$ and $J(t_uu)$ is the (unique) strict global maximum of $t\mapsto J(tu)$ on $[0,\infty)$.
Remarking in addition that $J(tu)\to\infty$ as $t\to\infty$ in the case where $\int_{\R^N}Q(x)|u(x)|^p\, dx=0$, the assertion follows.\\
(ii) Let $\delta>0$. For each $u\in\mN$, it follows from (i), that  $J(\frac{\delta}{\|u\|_V}u)\leq J(u)$ holds. Therefore,
$c=\infl_{u\in\mN}J(u)\geq \inf\{J(w)\, :\, \|w\|_V=\delta\}$ for any $\delta>0$. Now from the Sobolev embedding 
$H^s(\R^N)\embed L^p(\R^N)$,
there is some constant $C>0$ such that for all $w\in H^s(\R^N)$ with $\|w\|_V=\delta$,
$$
J(w)\geq \frac12\|w\|_V^2-\frac1pC^p\|Q\|_\infty\|w\|_V^p=\delta^2\left(\frac12-\delta^{p-2}\frac{C^p\|Q\|_\infty}{p}\right).
$$
Since the last expression is positive for sufficiently small $\delta$, we obtain $c>0$. 
To prove the second and third statements in (ii) as well
as the property (iii), it suffices to notice that if $u\in\mN$, then
$c\leq J(u)=\left(\frac12-\frac1p\right)\|u\|_V^2\leq\left(\frac12-\frac1p\right)\|Q\|_\infty\|u\|_{L^p}^p$ holds.
\end{proof}
Let us point out that the same result holds with $J_\infty$, $c_\infty$ and $\mN_\infty$ in place of $J$, $c$ and $\mN$, respectively.
Minimizers of $J_\infty$ on $\mN_\infty$ are critical points of $J_\infty$ on $H^s(\R^N)$ with least 
possible energy among all nontrivial critical points.
For this reason, they are called {\em ground states} of \eqref{eqn:frac_schroed_inf}.
According to \cite[Proposition 3.1]{FLS2013},  there is (up to translation) a  unique
ground state solution $u_\infty\in H^{2s+1}(\R^N)\cap\mC^\infty(\R^N)$ of
\eqref{eqn:frac_schroed_inf}, positive, radially symmetric, radially
decreasing and which satisfies the following asymptotic decay properties:
\begin{equation}\label{eqn:decay_u_infty}
\frac{C_1}{1+|x|^{N+2s}}\leq u_\infty(x)\leq\frac{C_2}{1+|x|^{N+2s}}\quad\text{for all }x\in\R^N,
\end{equation}
where $0<C_1< C_2$. Our next result states that critical points $u$ of $J$ with $J(u)\leq 2c$ cannot 
change sign, where we recall that $c$ is defined in \eqref{eq:c}.
\begin{prop}\label{prop:sign}
Let $u\in H^s(\R^N)$ be a nontrivial sign-changing critical point of $J$ (resp. $J_\infty$). Then  $J(u)>2c$ 
(resp. $J_\infty(u)>2c_\infty$).
\end{prop}
\begin{proof}
Let $u\in H^s(\R^N)$ be a nontrivial sign-changing critical point of $J$ and let $u_\pm=\max\{\pm u, 0\}$
denote the positive and negative part of $u$, respectively. Then $u_\pm\in H^s(\R^N)$ and $u=u_+-u_-$.
Moreover, letting $\ps{\cdot}{\cdot}_V$ denote the scalar product that induces the norm $\|\cdot\|_V$, the identity
$$
\ps{v}{w}_V=C_{N,s} \int\int_{\R^N\times\R^N}\frac{(v(x)-v(y))(w(x)-w(y))}{|x-y|^{N+2s}}\, dxdy + \int_{\R^N}V(x)vw\, dx,
$$
which holds for all $v,w\in H^s(\R^N)$, gives
$$
\ps{u_+}{u_-}_V=2C_{N,s}\ P.V.\int\int_{\{u>0\}\times\{u<0\}}\frac{u(x)u(y)}{|x-y|^{N+2s}}\, dxdy<0,
$$
for some $C_{N,s}>0$  and where $\text{P.V.}$ denotes the fact that the integral is taken in the `principal value' sense.
Now, from $0=J'(u)u_\pm=\ps{u}{u_\pm}_V\mp\int_{\R^N}Q(x)u_\pm^p\, dx$, we deduce that
$\int_{\R^N}Q(x)u_\pm^p\, dx=\|u_\pm\|_V^2-\ps{u_+}{u_-}_V>0$. Using \eqref{eqn:t_u} we can therefore find $t_\pm>0$, given
by
$$
t_\pm^{p-2} =\frac{\|u_\pm\|_V^2}{\int_{\R^N}Q(x)u_\pm^p\, dx}=\frac{\|u_\pm\|_V^2}{\|u_\pm\|_V^2-\ps{u_+}{u_-}_V}<1,
$$
such that $t_\pm u_\pm\in\mN$. Consequently,
$$
2c\leq J(t_+ u_+)+J(t_-u_-)<\left(\frac12-\frac1p\right) \left(\int_{\R^N}Q(x)u_+^p\,dx+\int_{\R^N}Q(x)u_-^p\, dx\right)=J(u),
$$
and thus $J(u)> 2c$ holds for any sign-changing critical point $u$ of $J$. The corresponding assertion for $J_\infty$ follows by
replacing $V(x)$ and $Q(x)$ by $1$ in the preceding arguments.
\end{proof}

\begin{cor}\label{cor:positive}
Let $u$ be  a nontrivial critical point of $J$ such that $J(u)\leq
2c$, where   $c$ is defined in \eqref{eq:c}. Then $u$ is positive in
$\R^N$
\end{cor}
\begin{proof}
By Proposition \ref{prop:sign}, $u\geq 0$   in $\R^N$. In  particular, we have
$\Ds u+V(x)u\geq 0$ in $\R^N$.
Since $u\neq 0$,  it follows from the strong maximum principle \cite[Corollary 3.4 and Remark 3.5]{FJ} that  $u>0$ in $\R^N$.
\end{proof}
We now give a first existence result in the case where the ground state energy level $c$ is strictly less than the ground state level of
the limit problem. In this case, we show that a ground state for \eqref{eqn:frac_schroed} exists, i.e., there exists a weak solution $u$ of
\eqref{eqn:frac_schroed} having minimal energy $J(u)=c$ among all nontrivial solutions. Notice that the conditions (V) and (Q) ensure
that $c\leq c_\infty$ holds in any case.
\begin{prop}\label{prop:c<cinf_f}
  Suppose that (V) and (Q) are satisfied.
  If $c<c_\infty$ holds, then $J$ has a critical point $u$ which satisfies $J(u)=c$. In particular \eqref{eqn:frac_schroed} has a positive
  solution.
\end{prop}
\begin{proof}
Let $(v_n)_n\subset \mN$ be a minimizing sequence for $J$.
Since $J$ is of class $C^2$, $J''(v)(v,v)=(2-p)\|v\|_V^2<0$ for all $v\in\mN$ and since $0$ is an isolated point of 
$\{u\in H^s(\R^N)\,:\, J'(u)u=0\}$,
we find that $\mN$ is a closed $C^1$-submanifold of codimension $1$ of the Hilbert space $H^s(\R^N)$ and, 
from Lemma \ref{lem:Nehari_f}, $J$ is bounded below on $\mN$. From Ekeland's variational principle 
(see e.g. \cite[Theorem 3.1]{ekeland74}), we can find a Palais-Smale sequence $(u_n)_n\subset\mN$ for $J$ at level
$c$ such that $\|u_n-v_n\|_V\to 0$ as $n\to\infty$. Lemma \ref{lem:Nehari_f} (iii) implies that $(u_n)_n$ is bounded 
and therefore, up to a subsequence, we may assume $u_n\weakto u$ weakly in $H^s(\R^N)$. 
Since $J'$ is weakly sequentially continuous, we find $J'(u)=0$. Now, if $u\neq 0$, then $u\in\mN$ and it follows that
$$
c\leq J(u)=\left(\frac12-\frac1p\right)\|u\|_V^2\leq\liminf_{n\to\infty}\left(\frac12-\frac1p\right)\|u_n\|_V^2=\liminf_{n\to\infty}J(u_n)=c.
$$
Hence $u$ is a critical point of $J$ at level $c$.
In the case where $u=0$ holds, we claim that we can find a sequence $(y_n)_n\subset\R^N$ and $\delta>0$ such that
\begin{equation}\label{eqn:liminf_delta_f}
  \liminf_{n\to\infty}\int_{B_1(0)} (y_n\ast u_n)^2\, dx \geq \delta>0.
\end{equation}
Indeed, if this were false, the concentration-compactness Lemma (see \cite{CZ_Rabinowitz92} and \cite[Lemma 2.2]{FQT2013})
would imply $\|u_n\|_{L^p}\to 0$ as $n\to\infty$ and therefore
$$
c=\llim_{n\to\infty} J(u_n)=\llim_{n\to\infty}\left(\frac12-\frac1p\right)\int_{\R^N}Q(x)|u_n(x)|^p\, dx =0,
$$
contradicting the fact that $c>0$.

Now, we remark that $(y_n)_n$ must be unbounded, since we are assuming $u_n\rightharpoonup 0$.
Hence, going to a subsequence, if needed, we can assume $|y_n|\to\infty$, $y_n\ast u_n\rightharpoonup w$ 
and $u_n(x-y_n)\to w(x)$ for a.e. $x\in\R^N$, as $n\to\infty$. 
The compact embedding $H^s(B_1(0))\hookrightarrow L^2(B_1(0))$ and \eqref{eqn:liminf_delta_f} together give $w\neq 0$.
Furthermore, Lemma \ref{lem:Nehari_f} (i) gives for every $t_n>0$,
\begin{align*}
  J(u_n)\geq J(t_nu_n) &= J_\infty(t_n(y_n\ast u_n)) +\frac{t_n^2}{2}\int_{\R^N}(V(x-y_n)-1)(y_n\ast u_n)^2\, dx \\
  &\quad -\frac{t_n^p}{p}\int_{\R^N}(Q(x-y_n)-1)|y_n\ast u_n|^p\, dx.
\end{align*}
Taking $t_n>0$ such that $t_n(y_n\ast u_n)\in\mN_\infty$ holds, we infer from \eqref{eqn:t_u} that 
$(t_n)_n$ is a bounded sequence, since $\int_{B_1(0)}|y_n\ast u_n|^p\, dx\to \int_{B_1(0)}|w|^p\, dx>0$ as $n\to\infty$.
Since $V(x-y_n)\to 1$, $Q(x-y_n)\to 1$ and $u_n(x-y_n)\to w(x)$ for a.e. $x\in\R^N$, as $n\to\infty$,
the dominated convergence theorem implies $\int_{\R^N}(V(x-y_n)-1)(y_n\ast u_n)^2\, dx\to 0$ and 
$\int_{\R^N}(Q(x-y_n)-1)|y_n\ast u_n|^p\, dx\to 0$, as $n\to\infty$.
We therefore conclude that
\[c=\lim_{n\to\infty} J(u_n)\geq \limsup_{n\to\infty} J_\infty(t_n (y_n\ast u_n))\geq c_\infty\]
holds, in contradiction to the assumption $c<c_\infty$. Thus, $u\neq 0$ must hold.  
Finally  by Corollary \ref{cor:positive} we have that  $u>0$ in $\R^N$.
\end{proof}
In \cite[Theorem 1.2]{FQT2013}, the authors give conditions under which $c<c_\infty$ is satisfied and therefore 
a ground state solution exists for \eqref{eqn:frac_schroed}.
On the contrary, if in addition to (V) and (Q) we require $V(x)\geq 1$ and $Q(x)\leq 1$ for a.e. $x\in\R^n$ 
with one of the inequalities being strict in a set of positive measure, then for all $u\in H^s(\R^N)$, 
$J(u)\geq J_\infty(u)$, which gives $c=c_\infty$. Moreover, assuming by contradiction that this energy level is attained, 
there would exist $u\in\mN$ satisfying $J(u)=c_\infty$, and choosing $\tau>0$ such that $\tau u\in\mN_\infty$ we would 
obtain by Lemma \ref{lem:Nehari_f}:
$$c_\infty=J(u)\geq J(\tau u)\geq J_\infty(\tau u)\geq c_\infty$$ and consequently $\tau=1$, {\em i.e.}, $u\in\mN_\infty$ and 
$J_\infty(u)=c_\infty$. The uniqueness (up to translations)
and the positivity of the ground state solution of \eqref{eqn:frac_schroed_inf} would then imply $u>0$ on $\R^N$ and therefore
\begin{align*}
J(u)-J_\infty(u)&=\left(\frac12-\frac1p\right)\int_{\R^N}(V(x)-1)u(x)^2\, dx\\
&=\left(\frac12-\frac1p\right)\int_{\R^N}(Q(x)-1)u(x)^p\, dx
\end{align*}
Since $V(x)>1$ or $Q(x)<1$ holds on a set of positive measure we would obtain $J(u)\neq J_\infty(u)$. 
This contradiction shows that the ground state level for $J$ is not attained.
In the sequel, we shall prove that in such a case a solution can still be found, provided conditions (H) and (U) are satisfied.

\section{Asymptotic estimates and Palais-Smale sequences}
\label{sec:asympt_estim}
The first part of this section is devoted to the proof of an energy estimate that will be an essential 
ingredient in the proof of Theorem \ref{thm:exist}.
As in \cite{BaLi90}, we consider convex combinations of two translates of the ground state solution $u_\infty$ and project them onto
the Nehari manifold $\mN$. We derive estimates concerning the energy of such a convex combination, showing that it can be made
smaller than $2c_\infty$, as each of the translates is moved away from the other and far away from the origin. Pointing out that the convex
combination of two translates of $u_\infty$ is an everywhere positive function, we can use \eqref{eqn:t_u} to define its projection onto $\mN$,
for which we introduce the following
\begin{notation}
Let $y,z\in\R^N$ and $\lambda\in[0,1]$. We denote by $t_\infty=t_\infty(\lambda, y, z)$ the unique positive number
(see Lemma \ref{lem:Nehari_f}) for which
$t_\infty[(1-\lambda)(y\ast u_\infty)+\lambda(z\ast u_\infty)]\in\mN$ holds.
\end{notation}

From now on, we will work   under the assumptions (V), (Q) and (H)
of Theorem \ref{thm:exist}. Furthermore, we will assume without loss of
generality that
\begin{equation}\label{eq:12}
 N+2s<\alpha<\min\{2(N+2s), \frac{p}{2}(N+2s)\}
\end{equation}
holds in (H).

Before stating and proving the above mentioned key energy estimate, we start by studying the asymptotic behavior 
of some integrals of convolution type.
\begin{lem}\label{lem:asympt_integral}
Let $\sigma, \tau\in(N,\infty)$. Setting $\mu=\min\{\sigma,\tau\}$, there holds
$$
\sup_{y\in\R^N}|y|^{\mu}\int_{\R^N}(1+|x|)^{-\sigma}(1+|x-y|)^{-\tau}\, dx<\infty.
$$
\end{lem}
\begin{proof}
Since $|y|/2\leq |x|$ whenever $|x-y|\leq |y|/2$, we have
\begin{align*}
|y|^\mu\int\limits_{\{|x-y|\leq\frac{|y|}{2}\}}(1+|x|)^{-\sigma}&(1+|x-y|)^{-\tau}\, dx
\leq|y|^\mu \frac{2^\sigma}{(2+|y|)^\sigma} \int_{\R^N}(1+|x-y|)^{-\tau}\, dx\\
&\leq  |y|^\mu\left(\frac{2}{2+|y|}\right)^\sigma
\int_{\R^N}(1+|x|)^{-\tau}\, dx.
\end{align*}
Next, we have
\begin{align*}
|y|^\mu\int\limits_{\{|x-y|\geq\frac{|y|}{2}\}}(1+|x|)^{-\sigma}&(1+|x-y|)^{-\tau}\,
dx\leq |y|^\mu\frac{2^\tau
}{(2+|y|)^\tau}\int_{\R^N}(1+|x|)^{-\sigma}\, dx.
\end{align*}
Since $\sigma, \tau\geq \mu >N$, the conclusion follows.
\end{proof}

\begin{lem}[Energy estimate]\label{lem:energy}
Suppose (V), (Q) and (H) hold. Then there exists $R_1>0$ such that
$$
J\bigl(t_\infty[(1-\lambda)(y\ast u_\infty) + \lambda(z\ast u_\infty)]\bigr)<2c_\infty
$$
for all $\lambda\in[0,1]$, $R\geq R_1$ and $y,z\in\R^N$
satisfying $|y| \ge R$, $|z|\geq R$ and $\frac23 R\leq |y-z|\leq 2R$.
\end{lem}
\begin{proof}
Let us consider
\begin{equation}\label{eq:7}
 R \ge 1\quad \text{and} \quad y,z \in \R^N \quad \text{with $|y| \ge R$, $|z|\geq R$ and $\frac23 R\leq |y-z|\leq 2R$.}
\end{equation}
For such $y,z$ and $\lambda\in[0,1]$, we set $w_\infty=(1-\lambda)(y\ast u_\infty)+\lambda(z\ast u_\infty)$ and choose 
$t_\infty=t_\infty(\lambda,y,z)$ as above, {\em i.e.}, such that $t_\infty w_\infty\in\mN$.

In the following, all constants will neither depend on $R$ nor on $\lambda,y,z$. In a first step, we will give bounds 
on various terms related to the energy functional, with respect to the following nonlinear interaction term:
$$
A_{y,z}:= \int_{\R^N} (y*u_\infty)^{p-1} (z * u_\infty)\,dx=\int_{\R^N} u_\infty^{p-1} ((z-y)* u_\infty)\,dx.
$$
Let us first remark that, since $u_\infty$ is positive and radially decreasing, the estimate \eqref{eqn:decay_u_infty} gives
\begin{equation}\label{eq:8_1}
A_{y,z}\geq \left(\frac{C_1}{2}\right)^{p-1}\int_{B_1(0)}u_\infty(x-(z-y))\, dx\geq \zeta_1|z-y|^{-(N+2s)}
\end{equation}
for all $|z-y|>1$ and some constant $\zeta_1>0$.

On the other hand, from \eqref{eqn:decay_u_infty} and Lemma \ref{lem:asympt_integral} with $\sigma=(p-1)(N+2s)$ and 
$\tau=N+2s$, there exists a constant $\zeta_2>0$ independent of $y, z$ such that
\begin{equation}\label{eq:8_2}
A_{y,z}\leq \zeta_2 |z-y|^{-(N+2s)}
\end{equation}
for all $|z-y|>1$. Let now $\alpha$ be as in assumption (H) and satisfy \eqref{eq:12}. Applying
Lemma \ref{lem:asympt_integral} with $\sigma=\alpha$ and $\tau=2(N+2s)$, we obtain
$$
\int_{\R^N} (1+|x|)^{-\alpha}\{(y\ast u_\infty)^2 + (z\ast u_\infty)^2\}\, dx
  \leq C \max \{|y|^{-\alpha},|z|^{-\alpha}\}
$$
for all $y,z \in \R^N$ with some constant $C>0$. Since $|y|, |z|\geq R\geq \frac12|z-y|$,
we have, making $\zeta_1$ larger if necessary,
\begin{equation}\label{eq:10}
  \int_{\R^N} (1+|x|)^{-\alpha}\{(y\ast u_\infty)^2 + (z\ast u_\infty)^2\}\, dx \leq \zeta_1^{-1} R^{N+2s-\alpha} A_{y,z}
\end{equation}
for all $R,y,z$ satisfying \eqref{eq:7}. A further application of Lemma \ref{lem:asympt_integral} with
$\sigma=\alpha$, $\tau=p(N+2s)$, gives
$$
  \int_{\R^N} (1+|x|)^{-\alpha}\{(y\ast u_\infty)^p + (z\ast u_\infty)^p\}\, dx
  \le C' \max \{|y|^{-\alpha},|z|^{-\alpha}\}
$$
for all $y,z \in \R^N$, $|y|, |z|\geq 1$ with some constant $C'>0$. As above, making $\zeta_1$ again larger if necessary,
we can write
\begin{equation} \label{eq:13}
  \int_{\R^N} (1+|x|)^{-\alpha}\{(y\ast u_\infty)^p + (z\ast u_\infty)^p\}\, dx \le \zeta_1^{-1} R^{N+2s-\alpha} A_{y,z}
\end{equation}
for all $R,y,z$ satisfying \eqref{eq:7}. Finally, Lemma \ref{lem:asympt_integral} with
$\sigma=\tau=\frac{p}{2}(N+2s)$ yields
$$
\int_{\R^N} (y\ast u_\infty)^{\frac{p}{2}}(z\ast u_\infty)^{\frac{p}{2}}\, dx=
  \int_{\R^N} u_\infty^{\frac{p}{2}}((y-z) \ast u_\infty)^{\frac{p}{2}} \, dx \le C'' |y-z|^{-\frac{p}{2}(N+2s)}
$$
for all $y,z \in \R^N$ with some constant $C''>0$. Therefore, making $\zeta_1$ again larger if necessary we find, using \eqref{eq:12},
\begin{equation} \label{eq:14}
  \int_{\R^N} (y\ast u_\infty)^{\frac{p}{2}}(z\ast u_\infty)^{\frac{p}{2}}\, dx \leq \zeta_1^{-1} R^{N+2s-\alpha} A_{y,z}
\end{equation}
holds for all $R,y,z$ satisfying \eqref{eq:7}. We now have all the tools to estimate
\[ J(t_\infty w_\infty) = \frac{t_\infty^2}{2}\|w_\infty\|_V^2 -\frac{t_\infty^p}{p}\int_{\R^N}Q(x)(w_\infty(x))^p\, dx.\]
We start by estimating the term $\|w_\infty\|_V^2$ which we split in the following way.
\begin{align*}
  &\int_{\R^N}|\xi|^{2s}|\what{w}_\infty(\xi)|^2\, d\xi+\int_{\R^N}V(x)w_\infty^2\, dx\\
  &\quad =((1-\lambda)^2+\lambda^2)\left(\int_{\R^N}|\xi|^{2s}|\what{u}_\infty|^2\, d\xi+\int_{\R^N} u_\infty^2\, dx\right)\\
  &\quad+ 2(1-\lambda)\lambda\left(\int_{\R^N}|\xi|^{2s}\text{Re}\bigl(e^{-i(y-z)\cdot\xi}\what{u}_\infty(\xi)\overline{\what{u}_\infty(\xi)}\bigr)\, d\xi
  +\!\!\int_{\R^N}(y\ast u_\infty)(z\ast u_\infty)\, dx\right)\\
  &\quad+\int_{\R^N}(V(x)-1)w_\infty^2\, dx.
\end{align*}
The property $J_\infty'(u_\infty)=0$ implies that
\begin{equation}\label{eqn:Jprime_infty}
\begin{aligned}
\int_{\R^N}|\xi|^{2s}&\text{Re}\left(e^{-i(y-z)\cdot\xi}\what{u}_\infty(\xi)\overline{\what{u}_\infty(\xi)}\right)\, d\xi+\int_{\R^N}(y\ast u_\infty)(z\ast u_\infty)\, dx\\
  &=\int_{\R^n}(y\ast u_\infty)^{p-1}(z\ast u_\infty)\, dx= A_{y,z}=A_{z,y}.
 \end{aligned}
\end{equation}
We deduce from \eqref{eq:10} and condition (H) that
\begin{align*}
  &\int_{\R^N}(V(x)-1) |w_\infty|^2\, dx
   \leq 2\widetilde{\kappa}_1 \int_{\R^N}(1+|x|)^{-\alpha}[(y\ast u_\infty)^2 + (z\ast u_\infty)^2]\, dx\\
  &\leq 2 \widetilde{\kappa}_1 \zeta_1^{-1} R^{N+2s-\alpha} A_{y,z}\qquad \text{for all $\lambda\in[0,1]$ and $R,y,z$ satisfying \eqref{eq:7},}
\end{align*}
where $\widetilde{\kappa}_1=2^\alpha\max\{\kappa_1,\|V\|_\infty\}$.
In addition, we point out that from the condition (V), there follows
\begin{equation}\label{eqn:conv_V}
\|w_\infty\|_V^2\longrightarrow ((1-\lambda)^2+\lambda^2)\|u_\infty\|_s^2\quad\text{ as }R\to\infty,
\end{equation}
uniformly in $\lambda\in[0,1]$.
Turning to the second integral, we write
\begin{align*}
  \int_{\R^N}Q(x)(w_\infty(x))^p\, dx &= (\lambda^p+(1-\lambda)^p)\int_{\R^N} u_\infty^p\, dx\\
  &+\int_{\R^N} w_\infty^p -[ \lambda^p(y\ast u_\infty)^p+(1-\lambda)^p(z\ast u_\infty)^p ]\, dx \\
  &+\int_{\R^N}(Q(x)-1) (w_\infty(x))^p\, dx.
\end{align*}
From \cite[Lemma 2.1]{BaLi90}, there exists a constant $C=C(p)>0$ such that for all $a, b\geq 0$
\begin{equation}
(a+b)^p\geq a^p+b^p+p(a^{p-1}b+ab^{p-1})-Ca^{\frac{p}{2}}b^{\frac{p}{2}}.
\end{equation}
Hence, from \eqref{eq:14}, we obtain
\begin{align*}
  &\int_{\R^N}w_\infty^p -[ \lambda^p(y\ast u_\infty)+(1-\lambda)^p(z\ast u_\infty) ]\, dx\\
  &\geq p\lambda^{p-1}(1-\lambda)\int_{\R^N}u_\infty(x-y)^{p-1}u_\infty(x-z)\, dx\\
  &\qquad+p\lambda(1-\lambda)^{p-1}\int_{\R^N}u_\infty(x-y)u_\infty(x-z)^{p-1}\, dx\\
  &\qquad-C\lambda^{\frac{p}{2}}(1-\lambda)^{\frac{p}{2}}\int_{\R^N}u_\infty(x-y)^{\frac{p}{2}}u_\infty(x-z)^{\frac{p}{2}}\, dx\\
  &\geq \bigl[p(\lambda^{p-1}(1-\lambda)+\lambda(1-\lambda)^{p-1})-C\zeta_1^{-1} R^{N+2s-\alpha}\bigr]A_{y,z}
\end{align*}
for $\lambda\in[0,1]$ and $R,y,z$ satisfying \eqref{eq:7}. Moreover, condition (H) as well as
\eqref{eq:13}  imply that, setting $\widetilde{\kappa}_2=2^\alpha\max\{\kappa_2,\|Q\|_\infty\}$,
\begin{align*}
  \int_{\R^N}(Q(x)-1) (w_\infty(x))^p\, dx
 & \geq -2^{p-1}\widetilde{\kappa}_2 \int_{\R^N} (1+|x|)^{-\alpha}\{(y\ast u_\infty)^p + (z\ast u_\infty)^p\}\, dx\\
 & \geq -2^{p-1}\widetilde{\kappa}_2\zeta_1^{-1} R^{N+2s-\alpha} A_{y,z},
\end{align*}
for $\lambda\in[0,1]$ and $R,y,z$ satisfying \eqref{eq:7}. Let us also remark that the assumption (Q) ensures
\begin{equation}\label{eqn:conv_Q}
\int_{\R^N}Q(x)(w_\infty(x))^p\, dx \longrightarrow ((1-\lambda)^p+\lambda^p)\int_{\R^N}u_\infty^p\, dx\quad\text{as }R\to\infty,
\end{equation}
uniformly in $\lambda\in[0,1]$. Combining this last inequality with \eqref{eqn:t_u} and \eqref{eqn:conv_V}, we find
\begin{equation}\label{eqn:conv_t_inf}
0<t_\infty=\left(\frac{\|w_\infty\|_V^2}{\int_{\R^N}Q(x)(w_\infty(x))^p\, dx}\right)^{\frac{1}{p-2}}
\longrightarrow \left(\frac{(1-\lambda)^2+\lambda^2}{(1-\lambda)^p+\lambda^p}\right)^{\frac{1}{p-2}},
\end{equation}
as $R\to\infty$, uniformly in $\lambda\in[0,1]$.
Since
\begin{align*}
&J(t_\infty w_\infty)=\left(\frac12-\frac1p\right) t_\infty^2\|w_\infty\|_V^2
\quad \leq \left(\frac12-\frac1p\right) t_\infty^2\Bigl[((1-\lambda)^2+\lambda^2)\|u_\infty\|_s^2\\
&\qquad+2\lambda(1-\lambda)A_{y,z}+2\widetilde{\kappa}_1\zeta_1^{-1}R^{N+2s-\alpha}A_{y,z} \Bigr]
\end{align*} and $c_\infty=\left(\frac12-\frac1p\right)\|u_\infty\|_s^2$,
we can find some $R_0\geq 1$ and
$0<\delta_0<1$ such that
\begin{equation}\label{eq:17}
J(t_\infty w_\infty) \leq \frac{3}{2}c_\infty
\end{equation}
for all $\lambda\in[0,\delta_0)\cup(1-\delta_0,1]$ and $R,y,z$
satisfying \eqref{eq:7} with $R\geq R_0$.

\bigskip

\noindent On the other hand, the above estimates together give for $R, y, z$ satisfying \eqref{eq:7} and $\lambda\in[\delta_0,1-\delta_0]$:
\begin{equation}\label{eq:15}
\begin{aligned}
  &J(t_\infty w_\infty)-2c_\infty \leq  J(t_\infty w_\infty)- J_\infty(\lambda t_\infty u_\infty) - J_\infty((1-\lambda)t_\infty u_\infty) \\
  &\qquad\leq -t_\infty^p\lambda(1-\lambda)\Bigl\{(1-\lambda)^{p-2}+\lambda^{p-2}-t_\infty^{2-p} -\zeta_3^{-1} R^{N+2s-\alpha}\Bigr\}A_{y,z},
\end{aligned}
\end{equation}
for some constant $\zeta_3>0$, where we used the fact that $t_\infty$ is bounded below away from $0$, uniformly in $y, z$ and $\lambda$.
According to \eqref{eqn:conv_t_inf} and since $\alpha>N+2s$, there exists $R_1\geq R_0$ such that
$$
(1-\lambda)^{p-2}+\lambda^{p-2}-t_\infty^{2-p}-\zeta_3^{-1}R^{N+2s-\alpha}\geq \delta_0^p
$$
for all $\lambda\in[\delta_0, 1-\delta_0]$ and $R\geq R_1$. With this choice, \eqref{eq:15} gives
$$
J(t_\infty w_\infty)\leq 2c_\infty -\kappa_3\delta_0^{p+2}A_{y,z}
$$
for all $\lambda\in[\delta_0,1-\delta_0]$ and $R\geq R_1$, with some constant $\kappa_3>0$. This, together with \eqref{eq:17} shows
$J(t_\infty w_\infty)<2c_\infty$ for every $\lambda\in[0,1]$ and $R, y, z$ satisfying \eqref{eq:7}
with $R\geq R_1$.
\end{proof}

The last result in this section describes the behavior of the bounded
Palais-Smale sequences for $J$, and shows that the same kind of
splitting as in the case $s=1$ studied by Benci and Cerami
\cite{benci-cerami87} (see also \cite[Proposition II.1]{BaLions97}
or \cite[Theorem 8.4]{WILLEM}) also occurs in the fractional case
$0<s<1$.
\begin{lem}\label{lem:Benci-Cerami}
  Let $(u_n)_n\subset H^s(\R^N)$ be a bounded sequence, for which $J'(u_n)\to0$ as $n\to\infty$.
  Then, there exist $\ell\in\N\cup\{0\}$, sequences $(x_n^i)_n\subset\R^N$, $1\leq i\leq \ell$, and $\overline{u}, w_1,\ldots, w_\ell\in H^s(\R^N)$ satisfying
  (up to a subsequence)
  \begin{itemize}
    \item[(i)] $J'(\overline{u})=0$,
    \item[(ii)] $J_\infty'(w_i)=0$, $i=1, \ldots, \ell$,
    \item[(ii)] $|x_n^i|\to\infty$ and $|x_n^i-x_n^j|\to\infty$ as $n\to\infty$ for $1\leq i\neq j\leq \ell$,
    \item[(iii)] $\bigl\|u_n-[\overline{u}+\suml_{i=1}^\ell x_n^i\ast w_i]\bigr\|_s\to 0$ as $n\to\infty$ and
    \item[(iv)] $J(u_n) \to J(\overline{u})+\suml_{i=1}^\ell J_\infty(w_i),$ as $n\to\infty$.
  \end{itemize}
\end{lem}
\begin{proof}
Since $(u_n)_n$ is a bounded sequence in $H^s(\R^N)$ we may assume, up to a subsequence, that
$u_n\weakto\overline{u}$ for some $\overline{u}\in H^s(\R^N)$, and the weak sequential continuity of $J'$ 
implies $J'(\overline{u})=0$.

\medskip

\textbf{Step 1}: Let $v_n^1:=u_n-\overline{u}$ for all $n\in\N$. Since $v_n^1\weakto 0$ in $H^s(\R^N)$ and 
since $V(x)\to 1$, resp. $Q(x)\to 1$ for $|x|\to\infty$, the compact embeddings $H^s(B_R(0))\hookrightarrow L^q(B_R(0))$, 
$R>0$, $2\leq q<2_s^\ast$ (resp. $2\leq q<\infty$ if $N=1$ and $s\geq\frac12$) gives
\begin{equation}\label{eqn:VQ}
  \int_{\R^N}(V(x)-1) |v_n^1|^2\, dx \to 0\quad\text{ and }\int_{\R^N}(Q(x)-1)|v_n^1|^p\, dx \to 0 \quad\text{ as }n\to\infty,
\end{equation}
Therefore, as $n\to\infty$, we find (up to a subsequence):
$J_\infty(v^1_n)=J(v^1_n) + o(1)$
\[= J(u_n) - J(\overline{u}) +\frac1p\int_{\R^N}Q(x) \bigl(|u_n|^p-|\overline{u}|^p-|v_n^1|^p\bigr)\, dx + o(1)
    = J(u_n)-J(\overline{u})+o(1),\]
where in the last step we have used the Br\'ezis-Lieb lemma \cite{brezis-lieb83}.
For $\varphi\in H^s(\R^N)$ with $\|\varphi\|=1$, we obtain moreover that
\begin{align*}
  J_\infty'(v^1_n)\varphi- J'(u_n)\varphi &+ J'(\overline{u})\varphi= \int_{\R^N}(1-V(x))v_n^1 \varphi\, dx\\ 
  &+ \int_{\R^N} (Q(x)-1)|v_n^1|^{p-2}v_n^1\varphi\, dx\\
 & +\int_{\R^N} Q(x)[ |u_n|^{p-2}u_n-|\overline{u}|^{p-2}\overline{u}-|v_n^1|^{p-2}v_n^1]\varphi\, dx
\end{align*}
tends uniformly to $0$ as $n\to\infty$, using \eqref{eqn:VQ} and a similar argument as in \cite[Lemma A.2]{BaLions97}.
Since $J'(u_n)\to 0$ as $n\to\infty$ and $J'(\overline{u})=0$, we find
\[J'_\infty(v^1_n)\to 0,\quad \text{ as }n\to\infty.\]
\textbf{Step 2}: Let \[\zeta=\limsup_{n\to\infty}\left(\sup_{y\in\R^N}\int_{B_1(y)}\bigl(v_n^1(x)\bigr)^2\, dx\right).\]
If $\zeta=0$, then the concentration-compactness Lemma gives $\|v_n^1\|_{L^p}\to 0$ as $n\to\infty$, and we infer
$$
 \|u_n-\overline{u}\|_s^2=\|v_n^1\|_s^2 = J_\infty'(v_n^1)v_n^1 + \int_{\R^N} |v_n^1|^p\, dx\to 0, \quad\text{as }n\to\infty.
$$
Hence $u_n\to \bar{u}$ in $H^s(\R^N)$ as $n\to\infty$, and the proof is complete.
In the case where $\zeta>0$, passing to a subsequence, we can find a sequence $(x_n^1)_n\subset\R^N$ satisfying
$|x_n^1|\to\infty$ as $n\to\infty$ and
\[\int_{B_1(0)} \bigl(v_n^1(x+x_n^1)\bigr)^2\, dx=\int_{B_1(x_n^1)}\bigl(v_n^1(x)\bigr)^2\, dx > \frac{\zeta}{2}\]
for all $n$. Since $((-x_n^1)\ast v_n^1)_n$ is bounded, going to a further subsequence if necessary, we obtain
$(-x_n^1)\ast v_n^1\rightharpoonup w_1\neq 0$, using the compactness of the embedding $H^s(B_1(0))\hookrightarrow L^2(B_1(0))$.
Furthermore, the weak sequential continuity of $J_\infty'$ and the invariance under translations of $J_\infty$ give 
for every $\varphi\in H^s(\R^N)$,
$$
J_\infty'(w_1)\varphi=\llim_{n\to\infty}J_\infty'\bigl((-x_n^1)\ast v_n^1\bigr)\varphi=\llim_{n\to\infty}J_\infty'(v_n^1)(x_n^1\ast\varphi)=0.
$$
Setting $v_n^2:= v_n^1-(x_n^1\ast w_1)$, it follows that $v_n^2\rightharpoonup 0$ in $H^s(\R^N)$, and the same 
arguments as above, applied to $J_\infty$, give
\[J_\infty(v_n^2)= J_\infty(v_n^1) -J_\infty(w_1)+o(1)=J(u_n)-J(\bar{u})-J_\infty(w_1)+o(1),\]
and $J_\infty'(v_n^2)\varphi-J_\infty'(v_n^1)\varphi+J_\infty'(w_1)(-x_n^1\ast\varphi)\to 0$ uniformly for $\|\varphi\|=1$, as $n\to\infty$.

\medskip

We conclude that $J_\infty'(v_n^2)\to 0$ as $n\to\infty$ and iterate the above procedure.
At each step we choose a sequence $(x_n^i)_n\subset\R^N$ such that $|x_n^i|\to\infty$ and $|x_n^i-x_n^j|\to\infty$ for all $i\neq j$, 
as $n\to\infty$, and obtain a critical point $w_i$ of $J_\infty$ such that with 
$v_n^{i+1}:=v_n^i-x_n^i\ast w_i=u_n-\overline{u}-\suml_{j=1}^ix_n^j\ast w_j$ there holds
(up to a subsequence)
$J_\infty(v_n^{i+1})= J(u_n)-J(\bar{u})-\suml_{j=1}^i J_\infty(w_j)+o(1)$, and $J_\infty'(v_n^{i+1})\to 0$
as $n\to\infty$.
Since $J_\infty(w)\geq c_\infty>0$ holds for every nontrivial critical point $w$ of $J_\infty$,
and since the boundedness of $(u_n)_n$ implies that $\sup\limits_{n\in\N}J(u_n)<\infty$,
the procedure has to stop after a finite number of steps.
\end{proof}

\section{Existence of a nontrivial solution}
\label{sec:existence_proof}
Assuming the conditions (V), (Q), (H) and (U), we now prove the existence of a nontrivial solution to \eqref{eqn:frac_schroed}, 
using the method of Bahri and Li \cite{BaLi90} (see also \cite{ew12}).
First note that $c\leq c_\infty$ holds, as can be deduced from \eqref{eqn:conv_V}, \eqref{eqn:conv_Q} and \eqref{eqn:conv_t_inf} 
by setting $\lambda=0$.
If $c<c_\infty$ then Proposition \ref{prop:c<cinf_f} gives the desired conclusion.

In the case $c=c_\infty$, we consider the barycenter map
$\beta$: $H^s(\R^N)\backslash\{0\}$ $\to$ $\R^N$ given by
\[\beta(u) = \frac{1}{\|u\|_{L^p}^p}\int_{\R^N}\frac{x}{|x|} |u(x)|^p\, dx, \quad u\in H^s(\R^N)\backslash\{0\}.\]
This mapping is continuous, and even uniformly continuous
on the bounded subsets of $H^s(\R^N)\backslash\{u\in H^s(\R^N)\, :\, \|u\|_{L^p}<r\}$ for any $r>0$.
Moreover, $|\beta(u)|<1$ for every $u\neq 0$.
For each $b\in B_1(0)\subset\R^N$ we now set
$$
I_b:=\inf_{\substack{u\in\mN\\ \beta(u)=b}}J(u) \geq c
$$
and distinguish two cases.

\noindent
{\bf Case 1}: $c=c_\infty=I_b$ for some $|b|<1$.\\
Here, we claim that $J$ has a nontrivial critical point at level $I_b=c=c_\infty$.
Indeed, let $(v_n)_n\subset\mN$ with $\beta(v_n)=b$ for all $n\in\N$ be a minimizing sequence for $I_b$. 
Since by Lemma \ref{lem:Nehari_f}, $(v_n)_n$ is bounded and $\mN$ is bounded away from $0$, 
we may choose, by the uniform continuity of $\beta$ on bounded subsets of $\mN$,
some $\delta>0$ such that for every $n\in\N$: $|\beta(v)|<\frac{1+|b|}{2}$ for all $v\in\mN$ with $\|v-v_n\|_s<\delta$.
According to Ekeland's variational principle, we can find a Palais-Smale sequence $(u_n)_n\in\mN$ for which
$J(u_n)\to I_b$ and $\|u_n-v_n\|_s\to 0$ holds, as $n\to\infty$. In particular, we find that $|\beta(u_n)|<\frac{1+|b|}{2}$
holds for $n$ large enough.
Assuming by contradiction that $J$ has no non-trivial critical point, the assumption
$c=c_\infty$, Lemma \ref{lem:Benci-Cerami} (iv) and the uniqueness of the ground state of \eqref{eqn:frac_schroed_inf} allow us to find,
going to a subsequence of $(u_n)_n$ if necessary,
a sequence $(x_n)_n\subset\R^N$ such that
$|x_n|\to\infty$ and $\|u_n-(x_n\ast u_\infty)\|_s\to 0$, as $n\to\infty$.
Since, also, $\|x_n\ast u_\infty\|_{L^p}=\|u_\infty\|_{L^p}>0$ holds for all $n$, and $|\beta(x_n\ast u_\infty)|\to 1$ as $n\to\infty$, the
uniform continuity of $\beta$ gives
\[1=\lim_{n\to\infty}|\beta(x_n\ast u_\infty)|\leq \limsup_{n\to\infty}|\beta(u_n)|\leq\frac{1+|b|}{2}.\]
This contradicts our assumption $|b|<1$, and therefore shows that $J$ has a critical point $u$ at level $c=c_\infty$.
This function $u$ is positive by   Corollary \ref{cor:positive}.

\noindent
{\bf Case 2}: $c=c_\infty<I_b$ for every $|b|<1$.\\
\noindent
In this case we will show that $J$ possesses a critical point  at some level $c_0\in [I_b, 2c )$ 
and Corollary \ref{cor:positive} will yield the conclusion.\\
For $R>0$, let $y=(0,\ldots,0,R)\in\R^N$ and consider the open ball
$$
\Omega_R:= B_{\frac43 R}(\textstyle{\frac{y}{3}})=\left\{ (1-\lambda)y + \lambda z\in\R^N\, :\, 0\leq \lambda< 1, \, z\in\partial\Omega_R\right\}.
$$
We define a min-max level $c_0$ as follows. Let $R\geq R_1$ where $R_1$ is given in Lemma \ref{lem:energy}, and consider
the projection of $\partial\Omega_R$ onto the Nehari manifold: $\gamma_0$: $\partial\Omega_R$ $\to$ $\mN$ given by
$$
\gamma_0(z):=t_\infty\cdot (z\ast u_\infty)\quad \text{ for all }z\in\partial\Omega_R,
$$
where $t_\infty=t_\infty(1,0,z)$ in the notation of Section \ref{sec:asympt_estim}.
We set $\ds\Gamma_R:=\{ \gamma\,:\, \overline{\Omega}_R \to \mN\, :\, \gamma\text{ continuous and } 
\gamma|_{\partial\Omega_R}=\gamma_0\}$
and consider the min-max energy level
\begin{equation}
  c_0:=\inf_{\gamma\in\Gamma_R}\max_{x\in\overline{\Omega}_R}J(\gamma(x)).
\end{equation}
We claim that for $b=(0,\ldots,0,|b|)$ with $0<|b|<1$ fixed, there holds
\begin{equation}
I_b\leq c_0< 2c_\infty
\end{equation}
for $R$ large enough.
To show the left-hand inequality, consider for each $\gamma\in\Gamma_R$ the homotopy
$\eta$: $[0,1]\times\overline{\Omega}_R$ $\to$ $\overline{B_1(0)}$ given by
$\eta(\theta, x)= \theta \beta(\gamma(x)) + (1-\theta) g(x)$,  $0\leq \theta\leq 1$, $x\in\overline{\Omega}_R$, where
$g$ is the homothetic contraction of $\overline{\Omega_R}$ onto the closed unit ball in $\R^N$:
$$
g(x)=\left\{\begin{array}{cc}\ds\tau(x)\frac{x}{|x|} & \text{if } x\neq 0 \\ \\ 0 & \text{if }x=0\end{array}\right. \text{ with}\quad
    \tau(x)=\frac{1}{5R}\left[ \sqrt{15|x|^2+ x_N^2}- x_N\right], \quad x\in\overline{\Omega_R}.
$$
Since $\gamma|_{\partial\Omega_R}=\gamma_0$ and $\theta\beta(\gamma_0(z))+(1-\theta) g(z)\to \frac{z}{|z|}$ uniformly
for $z\in\partial\Omega_R$ and $0\leq \theta\leq 1$, as $R\to\infty$, we obtain $b\notin \eta([0,1]\times\partial\Omega_R)$
for $R$ large enough.
The homotopy invariance of the degree then implies $\text{deg}(\beta\circ \gamma,\Omega_R,b)=\text{deg}(g,\Omega_R,b)=1$.
Using the existence property, we can therefore find some $x_b\in\Omega_R$ for which $\beta(\gamma(x_b))=b$,
and this gives $I_b\leq J(\gamma(x_b))$. Since $\gamma\in\Gamma_R$ was arbitrarily chosen, we obtain $I_b\leq c_0$.
Lemma \ref{lem:energy} gives the second inequality, when we consider $\gamma_2\in\Gamma_R$ given by
$$
\gamma_2((1-\lambda)y + \lambda z)= t_\infty((1-\lambda)(y\ast u_\infty) 
+ \lambda(z\ast u_\infty)),\qquad \lambda\in[0,1],\; z\in\partial\Omega_R.
$$
In particular, the min-max. level $c_0$ satisfies
\begin{equation}\label{eqn:crit_level}
  c=c_\infty< c_0 <2c_\infty
\end{equation}
for $R$ large enough.

Next, we point out that $\mN$ is a closed connected $C^1$-submanifold of the Banach space $H^s(\R^N)$. 
Moreover, for $R\geq R_1$, the the family $\mF_R=\{\gamma(\overline{\Omega}_R)\subset\mN\, :\, \gamma\in\Gamma_R\}$ 
of compact subsets of $\mN$ is a homotopy-stable family with boundary $\gamma_0(\partial\Omega_R)\subset\mN$, in the sense
of Ghoussoub \cite[Definition 3.1]{GHOU}.
Since $J(\gamma_0(z))$ converges to $c_\infty$ as $R\to\infty$, uniformly for $z\in\partial\Omega_R$, we have furthermore
$$
\max_{z\in\partial\Omega_R}J(\gamma_0(z))< c_0=\inf_{\gamma\in\Gamma_R}\max_{x\in\overline{\Omega}_R}J(\gamma(x))
  =\inf_{A\in\mF_R}\sup_{v\in A}J(v)
$$
for large $R$, and the min-max principle \cite[Theorem 3.2]{GHOU}, gives the existence of a Palais-Smale sequence
$(u_n)_n\subset\mN$ for $J$ at level $c_0$. Since $J$ is coercive on $\mN$, $(u_n)_n$ is bounded in $H^s(\R^N)$.
Moreover, since $J_\infty(w)>2 c_\infty$ holds for every sign-changing critical point $w$ of $J_\infty$
(see  Proposition \ref{prop:sign}), the estimate \eqref{eqn:crit_level}, the assumption (U)
and Lemma \ref{lem:Benci-Cerami} imply that, up to a subsequence, $u_n\to \overline{u}$ as $n\to\infty$ for some critical point
$\overline{u}\neq 0$ of $J$ which satisfies $J(\overline{u})=c_0<2c=2c_\infty$.
This concludes the proof.
$\qed$


\end{document}